\documentclass
{Uzhgorod-Mathematical-Paper-2025}
\umpnumber{47} \umpsubnumber{2} \umpyear{2025} \umpissnprint{2616-7700} \umpissnonline{2708-9568}\umpdoi{10.24144/2616-7700}

\begin{document}
\selectlanguage{english}
\udk{519.21}

\author{O.~Volkov} 
\institution{University of California, Berkeley} 
\position{Master's student at the Department of Statistics} 
\mail{oleksandr\_volkov@berkeley.edu}
\orcid{0009-0004-0247-9921}

\author{N.~Voinalovych} 
\institution{Volodymyr Vynnychenko Central Ukrainian State University} 
\position{Associate Professor of the Department of Mathematics and Digital Technologies} 
\academictitle{Candidate of Pedagogical Sciences} 
\mail{vojnalovichn@gmail.com}
\orcid{0000-0002-0523-7889}

\umpsuperchapter{Розділ 1: Математика і статистика}{Chapter 1: Mathematics and Statistics} 

\shortauthor{Волков~О., Войналович~Н.} 
\capitalauthor{O.~VOLKOV, N.~VOINALOVYCH} 
\authoreng{Volkov~O., Voinalovych~N.} 

\title{Про розподіл степеневого ряду з параметризацією середнім} 
\shorttitle{ON A POWER SERIES DISTRIBUTION WITH MEAN PARAMETERIZATION} 
\titleeng{On a Power Series Distribution with Mean Parameterization}
\maketitle
\received{16.09.2025}

\newcommand{\hm}[1]{#1\nobreak\discretionary{}{\hbox{\ensuremath{#1}}}{}} 

\begin{abstract}
\selectlanguage{english}
\smallindent The article examines the distribution of the power series of the function $w\left(y\right)=\left(1+\sqrt{1-y} \right)^{-\frac{1}{2} } $. The distribution of the considered function into a power series is obtained $\left(1+\sqrt{1-y} \right)^{-\frac{1}{2} } =\sum _{m=0}^{\infty }\frac{\left(4m\right)!16^{-m} }{\left(2m\right)!\left(2m+1\right)\sqrt{2} }  y^{m} $. The dispersion function is found $v\left(x\right)=x\left(2x+1\right)\left(4x+1\right)$, $x>0$. A distribution with mean parameterization is constructed
\[\Pr \left(\xi =k\right)=\begin{pmatrix} {4k+1} \\ {2k} \end{pmatrix}2^{-k} x^{k} \left(2k+1\right)^{k+\frac{1}{2} } \left(4k+1\right)^{-2k-\frac{3}{2} },\ \ x>0.\]
It is proved that the raw moments $\alpha _{m} $, central moments $\mu _{m} ,$ cumulants $\chi _{m} $, $m=1,\; 2,\; \ldots$ satisfy the following recurrence relations: $\alpha _{m+1} =x\alpha _{m} +v\left(x\right)\frac{d\alpha _{m} }{dx},$ $\alpha _{0} =1,\; \alpha _{1} =x$; $\mu _{m+1} =m\mu _{m-1} +v\left(x\right)\frac{d\mu _{m} }{dx},$ $\mu _{0} =1,\; \mu _{1} =0$; $\chi _{m+1} =v\left(x\right)\frac{d\chi _{m} }{dx} ,$ $\chi _{1} =x$.


\end{abstract}

\noindent
\paragraph{Introduction} The study of power-series distributions was initially undertaken in the classical work of A.~Noack [1], wherein the notion of random variables with discrete distributions defined via power series was introduced. Further development of this direction was reflected in the works of N.~L.~Johnson, A.~W.~Kemp, and S.~Kotz [2, 3], who, in their monographs, systematized the known discrete distributions and examined them within the unified framework of power-series distributions. A significant contribution to the study of generalizations of such distributions was made by P.~S.~Consul [4] and his co-authors, who developed the theory of generalized Poisson distributions and examined a wide class of families defined by power series.

In Ukraine, research in this domain has been advanced by the contributions of Yu.~I.~Volkov [5], who investigated power-series distributions under prescribed covariance structures, as well as by the textbook authored by Yu.~I.~Volkov and N.~M.~Voinalovych [6], which provides a systematic exposition of methodologies for the construction and analysis of such distributions. Particular attention is devoted to mean parameterization, which proves convenient for analytical representation and for establishing recurrence relations for the numerical characteristics of the distribution.

Prior findings have established that mean parameterization constitutes a convenient and informative framework for the investigation of discrete distributions, as it facilitates a direct correspondence between numerical characteristics and the expected value of the random variable. This, in turn, provides opportunities for a more effective analysis of the variance function and for the derivation of recurrence relations for moments and cumulants.

The significance of the present study arises from the necessity to broaden the class of established power-series distributions, to identify novel generating functions with positive coefficients, and to establish their properties under mean parameterization. These contributions are of importance not only from a theoretical standpoint -- advancing probability theory through the provision of new distributional examples and methodological tools of analysis --- but also from an applied perspective, as the resulting findings offer potential utility in the mathematical modeling of stochastic processes, in statistical machine learning for forecasting, and in a wide range of problems in applied mathematics.

\paragraph{Theoretical foundations of the study}

\begin{definition} Let 
\[w\left(y\right)=\sum _{k=0}^{\infty }a_{k} y^{k},\ \ 0<y<R,\] 
and all coefficients of this series are non-negative numbers. The distribution of an integer random variable $\xi $, given by the formula
\[p_{k} \left(y\right)=\Pr \left(\xi =k\right)=\frac{a_{k} y^{k} }{w\left(y\right)},\ \ k=0, 1, 2, \ldots,\] 
is called the power series distribution of the function$w\left(y\right)$ with parameter y.
\end{definition}

To obtain the numerical characteristics of the distribution, generating functions are used:

\begin{enumerate}[label=\arabic*)]
\item ordinary generating function
\[P(z):=\sum _{k=0}^{\infty }p_{k} \left(y\right)z^{k}  =\frac{1}{w\left(y\right)} \sum _{k=0}^{\infty }a_{k} y^{k} z^{k}  =\frac{w\left(yz\right)}{w\left(y\right)} ,\] 
\item generating function of raw moments 
\[A\left(z\right)=P\left(e^{z} \right), \] 
\item generating function of central moments 
\[C\left(z\right)=e^{-\alpha _{1} z} A\left(z\right),\ \ (\alpha _{1} \textrm{ --- expected value } \xi),\]

\item generating function of cumulants
\[K\left(z\right)=\log A\left(z\right).\] 
\end{enumerate}

Then expected value $E\xi =\frac{yw'\left(y\right)}{w\left(y\right)} $, variance $D\xi =y\frac{dE\xi }{dy} $.

Let $E\xi $ be denoted by \textit{x}. Then the function $x=\frac{yw'\left(y\right)}{w\left(y\right)} $ on the interval $\left(0,R\right)$ has an inverse $y=y\left(x\right)$, because $\frac{dx}{dy} =\frac{D\xi }{y} >0$. Therefore, it is possible to study the power series distribution with a different parameterization, specifically by characterizing the distribution with the parameter \textit{x}. This parameterization is called parameterization by the mean (see, for example, [2], p. 670). With this parameterization, we have 
\[\Pr \left(\xi =k\right)=p\left(k,x\right)=p_{k} \left(y\left(x\right)\right)=\frac{\left(y\left(x\right)\right)^{k} a_{k} }{w\left(y\left(x\right)\right)},\ \ k=0, 1, 2, \ldots,\] 
\[x\in X=\left(0,\; \frac{Rw'\left(R\right)}{w\left(R\right)} \right),\ \ D\xi =\frac{y}{\frac{dy}{dx} } =\frac{y\left(x\right)}{y'\left(x\right)} .\] 

The function $v\left(x\right)=\frac{y\left(x\right)}{y'\left(x\right)} $ is called the variance function of the distribution. 

\textbf{For example,}

\begin{itemize}[label=$\bullet$]
\item if $w\left(y\right)=1+y$, is $v\left(x\right)=x\left(1-x\right)$, $0<x<1$;

\item if $w\left(y\right)=e^{y} $, is $v\left(x\right)=x$, $x>0$;

\item if $w\left(y\right)=\left(y-1\right)^{-1} $, is $v\left(x\right)=x\left(1+x\right)$, $x>0$.
\end{itemize}

\paragraph{The main result} The goal of this work is to study the power series distribution of the function $w\left(y\right)=\left(1+\sqrt{1-y} \right)^{-\frac{1}{2} } $.

We will find the expansion of this function into a power series and verify that the coefficients of this series are non-negative. 

Since 
\[\frac{1-\sqrt{1-y} }{y} =\frac{1+\sqrt{y} }{2y} +\frac{1-\sqrt{y} }{2y} -2\sqrt{\frac{1-y}{4y^{2} } } =\left(\sqrt{\frac{1+\sqrt{y} }{2y} } -\sqrt{\frac{1-\sqrt{y} }{2y} } \right)^{2} \], 
will give us
\[\left(1+\sqrt{1-y} \right)^{-\frac{1}{2} } =\sqrt{\frac{1-\sqrt{1-y} }{y} } =\sqrt{\frac{1+\sqrt{y} }{2y} } -\sqrt{\frac{1-\sqrt{y} }{2y} } .\] 
Next, using the binomial series, we will have
\[\sqrt{1+\sqrt{y} } =\sum _{k=0}^{\infty }\frac{1}{k!}  \frac{1}{2} \left(\frac{1}{2} -1\right)\ldots\left(\frac{1}{2} -k+1\right)\left(\sqrt{y} \right)^{k} ,\] 
\[\sqrt{1-\sqrt{y} } =\sum _{k=0}^{\infty }\frac{1}{k!}  \frac{1}{2} \left(\frac{1}{2} -1\right)\ldots\left(\frac{1}{2} -k+1\right)\left(-\sqrt{y} \right)^{k} ,\] 
\[\sqrt{1+\sqrt{y} } -\sqrt{1-\sqrt{y} } =\sum _{m=0}^{\infty }\frac{2}{\left(2m+1\right)!}  \frac{1}{2} \left(\frac{1}{2} -1\right)\ldots\left(\frac{1}{2} -\left(2m+1\right)+1\right)\left(\sqrt{y} \right)^{2m+1} =\] 
\[=\sum _{m=0}^{\infty }\frac{\left(4m\right)!}{\left(2m\right)!\left(2m+1\right)16^{m} }  y^{m} \sqrt{y} .\] 

If this is divided by $\sqrt{2y} $, we will ultimately have
\[\left(1+\sqrt{1-y} \right)^{-\frac{1}{2} } =\sum _{m=0}^{\infty }\frac{\left(4m\right)!16^{-m} }{\left(2m\right)!\left(2m+1\right)\sqrt{2} }  y^{m} .\] 

Thus, the coefficients of this series are positive.

Let's find the variance function.

We have:
\[x=\frac{yw'\left(y\right)}{w\left(y\right)} =\frac{1-\sqrt{1-y} }{4\sqrt{1-y} } ,\] 
hence 
\[y=\frac{8x\left(2x+1\right)}{\left(4x+1\right)^{2} } .\] 
And hence 
\[v\left(x\right)=\frac{y\left(x\right)}{y'\left(x\right)} =x\left(2x+1\right)\left(4x+1\right),\ \ x>0.\] 
Then the distribution with parameterization by the mean will be such that
\[\Pr \left(\xi =k\right)=p\left(k,x\right)=\begin{pmatrix} 4k+1 \\ 2k \end{pmatrix}2^{-k} x^{k} \left(2k+1\right)^{k+\frac{1}{2} } \left(4k+1\right)^{-2k-\frac{3}{2} },\ \ x>0.\] 

Generating functions are used to obtain numerical characteristics.

It can be directly verified that the generating function $P(z)$ satisfies the differential equation
\begin{equation} \label{GrindEQ__1_} 
v\left(x\right)\frac{\partial P}{\partial x} -z\frac{\partial P}{\partial z} +xP=0,\ \ P\left(1\right)=1,\ x\in X.  
\end{equation} 

It follows that the generating function of the raw moments $A(z)$ satisfies the equation
\begin{equation} \label{GrindEQ__2_} 
v\left(x\right)\frac{\partial A}{\partial x} -\frac{\partial A}{\partial z} +xA=0,\ \ A\left(0\right)=1;  
\end{equation} 
the generating function of the central moments $C(z)$ satisfies the equation
\begin{equation} \label{GrindEQ__3_} 
v\left(x\right)\left(\frac{\partial C}{\partial x} +zC\right)-\frac{\partial C}{\partial z} =0,\ \ C\left(0\right)=1;  
\end{equation} 
the generating function of the cumulants $K(z)$ satisfies the equation
\begin{equation} \label{GrindEQ__4_} 
v\left(x\right)\frac{\partial K}{\partial x} -\frac{\partial K}{\partial z} +x=0,\ \ K\left(0\right)=0.  
\end{equation} 

\begin{theorem} The raw moments $\alpha _{m} $, central moments $\mu _{m} $, and cumulants $\chi _{m} ,$ $m=1, 2, \ldots$ satisfy the following recurrence relations:
\begin{equation} \label{GrindEQ__5_} 
\alpha _{m+1} =x\alpha _{m} +v\left(x\right)\frac{d\alpha _{m} }{dx} ,\ \ \alpha _{0} =1,\ \alpha _{1} =x,  
\end{equation} 
\begin{equation} \label{GrindEQ__6_} 
\mu _{m+1} =m\mu _{m-1} +v\left(x\right)\frac{d\mu _{m} }{dx} ,\ \ \mu _{0} =1,\ \mu _{1} =0,  
\end{equation} 
\begin{equation} \label{GrindEQ__7_} 
\chi _{m+1} =v\left(x\right)\frac{d\chi _{m} }{dx} ,\ \chi _{1} =x.  
\end{equation} 
\end{theorem}

\begin{proof} To obtain relation \eqref{GrindEQ__5_}, we differentiate \eqref{GrindEQ__2_} $m$ times with respect to $z$ and substitute $z=0$. To obtain relation \eqref{GrindEQ__6_}, we differentiate \eqref{GrindEQ__3_} $m$ times with respect to $z$ and substitute $z=0$. To obtain relation \eqref{GrindEQ__7_}, we differentiate \eqref{GrindEQ__4_} $m$ times with respect to $z$ and substitute $z=0$.

In particular, from \eqref{GrindEQ__5_}, \eqref{GrindEQ__6_}, \eqref{GrindEQ__7_} we find
\[\alpha _{2} =x^{2} +x\left(2x+1\right)\left(4x+1\right),\] 
\[\alpha _{3} =x^{3} +3x^{2} \left(2x+1\right)\left(4x+1\right)+v\left(x\right)v'\left(x\right),\] 
\[\chi _{2} =v\left(x\right),\] 
\[\chi _{3} =v\left(x\right)v'\left(x\right).\] 
\end{proof}

\paragraph{Conclusions} The power series distribution of the function $w\left(y\right)=\left(1+\sqrt{1-y} \right)^{-\frac{1}{2} } $ was obtained, and its characteristics were studied in the case of parameterization by the mean. The explicit form of the dispersion function was established, and recursive relations for the raw and central moments, as well as for the cumulants, were proved.

Unlike previously known results, which mainly considered individual examples or parameterization by the initial coefficients of the series, this work constructs the distribution using mean-based parameterization, allowing a direct connection between the properties of the distribution and its mathematical expectation. This approach simplifies the analysis of the behavior of the variance and moments and also enables effective comparison of different distributions within the same class.

The value of the obtained result lies in the fact that the considered generating function has positive coefficients in its power series expansion, ensuring the correctness of probability definitions and expanding the known class of power series distribution examples. Moreover, the established recursive relations for moments and cumulants can be used for further theoretical studies and applied computations in problems of mathematical statistics and machine learning for forecasting natural and social phenomena.

\selectlanguage{ukrainian}
\umpmaketitle

\begin{abstract}
\smallindent В статті досліджено розподіл степеневого ряду функції $w\left(y\right)=\left(1+\sqrt{1-y} \right)^{-\frac{1}{2} } $. Отримано розклад розглядуваної функції в степеневий ряд $\left(1+\sqrt{1-y} \right)^{-\frac{1}{2} } =$\linebreak
$=\sum _{m=0}^{\infty }\frac{\left(4m\right)!16^{-m} }{\left(2m\right)!\left(2m+1\right)\sqrt{2} }  y^{m} $. Знайдено дисперсійну функцію $v\left(x\right)=x\left(2x+1\right)\left(4x+1\right)$, $x>0$. Побудовано розподіл з параметризацією середнім
\[\Pr \left(\xi =k\right)=\begin{pmatrix} 4k+1 \\ 2k \end{pmatrix}2^{-k} x^{k} \left(2k+1\right)^{k+\frac{1}{2} } \left(4k+1\right)^{-2k-\frac{3}{2} },\ \ x>0.\]
Доведено, що початкові моменти $\alpha _{m} $, центральні моменти $\mu _{m},$ кумулянти\linebreak
$\chi _{m} $, $m=1, 2, \ldots$ задовольняють такі рекурентні співвідношення: $\alpha _{m+1} =x\alpha _{m} \hm+v\left(x\right)\frac{d\alpha _{m} }{dx}$, $\alpha _{0} =1, \alpha _{1} =x$; $\mu _{m+1} =m\mu _{m-1} +v\left(x\right)\frac{d\mu _{m} }{dx}$, $\mu _{0} =1, \mu _{1} =0$; $\chi _{m+1} =v\left(x\right)\frac{d\chi _{m} }{dx},$ $\chi _{1} =x$.
\end{abstract}

%
%
%
%
%

\end{document}